\documentclass[11pt]{article}
\usepackage{amsthm}
\usepackage{latexsym}
\usepackage[noblocks]{authblk}
\usepackage{amssymb,amsmath}  % use mathematical symbols
\usepackage{enumitem}
\usepackage{multimedia}
\usepackage{subfigure}
\usepackage[noadjust]{cite}
\usepackage{authblk} %%% new line between authors
\usepackage{mathrsfs} \usepackage{epsfig}
\usepackage{mathdots}
\usepackage{caption}
\usepackage{color}
\usepackage{pgfpages,tikz,tikz-cd,pgfkeys,pgfplots,calc}
\usetikzlibrary{shapes,snakes}
\usepackage{kotex}
\usepackage{titling}
\usepackage{setspace}
\usepackage{stmaryrd}
\usepackage[%utf8,
linesnumbered,ruled,vlined]{algorithm2e}
\usepackage{algorithmicx, algpseudocode, algcompatible}
\definecolor{LemonChiffon}{rgb}{100, 98, 80}
\definecolor{myblue}{rgb}{0,0.4,0.8}
\definecolor{orange}{rgb}{1, 0.4, 0}
\definecolor{mygreen}{rgb}{0, 0.8, 0.2}
\definecolor{myred}{rgb}{204, 0, 0}
\definecolor{violet}{RGB}{0.4,0.2,1}
\definecolor{brown}{rgb}{0.6, 0.4, 0}

\newtheorem{theorem}{Theorem}[section]
\newtheorem{lemma}[theorem]{Lemma}

\newtheorem{corollary}[theorem]{Corollary}

\theoremstyle{definition}
\newtheorem{definition}[theorem]{Definition}

\theoremstyle{definition}
\newtheorem{remark}[theorem]{Remark}
\newcounter{statement}

\begin{document}

\title{Structural properties of symmetric Toeplitz and Hankel matrices}
\author{Hojin Chu$^{1}$, Homoon Ryu$^{1}$ \\
 {\footnotesize $^{1}$ \textit{Department of Mathematics Education,
Seoul National University,}}\\{\footnotesize\textit{
Seoul 08826, Rep. of Korea}}\\
{\footnotesize\textit{
ghwls8775@snu.ac.kr, rhm95@snu.ac.kr}}\\
{\footnotesize}}
\date{}
\maketitle

\begin{abstract}
In this paper, we investigate properties of a symmetric Toeplitz matrix and a Hankel matrix by studying the components of its graph.
To this end, we introduce the notion of ``weighted Toeplitz graph" and ``weighted Hankel graph", which are weighted graphs whose adjacency matrix are a symmetric Toeplitz matrix and a Hankel matrix, respectively.
By studying the components of a weighted Toeplitz graph, we show that the Frobenius normal form of a symmetric Toeplitz matrix is a direct sum of symmetric irreducible Toeplitz matrices.
Similarly, by studying the components of a weighted Hankel matrix, we show that the Frobenius normal form of a Hankel matrix is a direct sum of irreducible Hankel matrices.
\end{abstract}



%C22-signed, weighted graphs
%c50-graphs and linear algebra
%c85-graph algorithm
%b05-Toeplitz, cauchy and related

\section{Introduction}
A {\it Toeplitz matrix} is a matrix in which each descending diagonal from left to right is constant.
That is, an $n\times n$ matrix $T=(t_{i,j})_{1\le i,j\le n}$ is a Toeplitz matrix if $t_{i,j}=t_{i+1,j+1}$ for each $1 \le i,j \le n-1$.
%Toeplitz matrices arise naturally in large variety of areas in pure and applied mathematics.
%For example, they arise in differential and integral equations, signal processing of engineering, numerical analysis, and statistics (see~\cite{T1_heinig1984algebraic}). 
%Indeed, Toeplitz matrices have a lot of significant characteristic properties as well as their attractive form.
%For these reasons, the study of Toeplitz matrices has been an active field (see~\cite{T_bottcher2013analysis,T_gray1972asymptotic, T_hansen2002deconvolution, T_reichel1992eigenvalues, T_tilli1998note, T_ye2016every}).
%
As a variation of the study of Toeplitz matrices in terms of graph theory, ``Toeplitz graph" was introduced and investigated with respect to hamiltonicity by van Dal {\it et al.}~\cite{G1_van1996hamiltonian} in 1966.
They defined a Toeplitz graph as a graph whose adjacency matrix is symmetric Boolean Toeplitz.
Results regarding Toeplitz graphs can be referenced in~\cite{G_cheon2023matrix,G_euler1995characterization,G_euler2013planar,G_heuberger2002hamiltonian,G_liu2019computing,G_mojallal2022structural,G_nicoloso2014chromatic}.
In this paper, we extend Toeplitz graphs to weighted Toeplitz graphs.

A {\it Hankel matrix} is a matrix in which each descending anti-diagonal from right  to left is constant.
That is, an $n\times n$ matrix $H=(h_{i,j})_{1\le i,j\le n}$ is a Hankel matrix if $h_{i+1,j}=h_{i,j+1}$ for each $1 \le i,j \le n-1$. 
We introduce the notion of ``weighted Hankel graph" as a variation of the study of Hankel matrices.

%Then one may verify that each Hankel matrix can be expressed as the product of an exchange matrix on both sides of a certain Toeplitz matrix, and vice versa.
%Several studies relate to Hankel matrices have been conducted together with Toeplitz matrices (see \cite{brualdi2014hankel,stroethoff1992hankel,lemmerling2001analysis, gu2003commutation,bryc2006spectral}).

Given two integers $m$ and $n$ with $m\le n$, let $\llbracket m,n \rrbracket$ be the set of integers between $m$ and $n$, that is,
\[ \llbracket m,n \rrbracket =\{k \in \mathbb{Z} \colon\, m \le k \le n\}. \]
We simply write $\llbracket 1,n \rrbracket$ as $[n]$.
Now, we define the concepts that are mainly used in this paper.
\begin{definition}\label{def}
	Given an $n\times n$ symmetric Toeplitz matrix 
	\[A:=\begin{bmatrix}
  a_0 & a_{1}   & a_{2} & \cdots & \cdots & a_{n-1} \\
  a_1 & a_0      & a_{1} & \ddots &        & \vdots \\
  a_2 & a_1      & \ddots & \ddots & \ddots & \vdots \\ 
 \vdots & \ddots & \ddots & \ddots & a_{1} & a_{2} \\
 \vdots &        & \ddots & a_1    & a_0    & a_{1} \\
a_{n-1} & \cdots & \cdots & a_2    & a_1    & a_0
\end{bmatrix}, \]
we may denote $A$ by $T[a_0, a_1, \ldots, a_{n-1}]$ since a symmetric Toeplitz matrix is uniquely determined by its first row.  
We let \[S_A = \{ i \in \llbracket 0,n-1 \rrbracket \colon\, a_i \neq 0 \} \]
and $w_A$ be the map from $S_A$ to $\mathbb{R}\setminus \{0\}$ defined by 
\[
w_A(s) = a_s.
\]

Let $A= T[a_0, a_1, \ldots, a_{n-1}]$. 
The edge-weighted graph $(G,w)$ is uniquely determined by $A$ in the following way:
\[V(G) = [n], \quad E(G) = \{ ij \colon\, i,j \in V(G), \ |i-j| \in S_A\},\quad  \text{and}\quad  w(ij) = w_A(|i-j|). \]
We denote this graph by $G(A)$. 
We mean by a {\it weighted Toeplitz graph} the edge-weighted graph of a symmetric Toeplitz matrix 
(see Figure~\ref{fig:wtex} for an illustration).
\end{definition}

\begin{definition}
	Given an $n\times n$ Hankel matrix 
	\[A:=\begin{bmatrix}
  a_0 & a_{1}   & a_{2} & \cdots & \cdots & a_{n-1} \\
  a_1 & a_2      &  & \iddots &   \iddots     & \vdots \\
  a_2 &       & \iddots & \iddots & \iddots & \vdots \\ 
 \vdots & \iddots & \iddots & \iddots &  & a_{2n-4} \\
 \vdots &  \iddots      & \iddots &     & a_{2n-4}    & a_{2n-3} \\
a_{n-1} & \cdots & \cdots & a_{2n-4}    & a_{2n-3}    & a_{2n-2}
\end{bmatrix}, \]
we may denote $A$ by $H[a_0, a_1, \ldots, a_{2n-2}]$ since a Hankel matrix is determined by the first row and the last column.  
We let \[T_A = \{ i \colon\, a_i \neq 0 \} \]
and $w_A$ be the map from $T_A$ to $\mathbb{R}$ defined by 
\[
w_A(t) = a_t.
\]
We simply write $T$ and $w$ instead of $T_A$ and $w_A$ if no confusion is likely. 

Let $A= H[a_0, a_1, \ldots, a_{2n-2}]$. 
The edge-weighted graph $(G,w)$ is uniquely determined by $A$ in the following way:
\[V(G) = [[0,n-1]], \quad E(G) = \{ ij \colon\, i,j \in V(G), \ i+j \in T_A\},\quad  \text{and}\quad  w(ij) = w_A(i+j). \]
We denote this graph by $G(A)$. 
A {\it weighted Hankel graph} means the edge-weighted graph of a Hankel matrix (see Figure~\ref{fig:hankel} for an illustration).
\end{definition}

In this paper, we investigate properties of symmetric Toeplitz and Hankel matrices by studying the components of their associated graphs.
We show that a component of a weighted Toeplitz graph is also a weighted Toeplitz graph and that there exists a preorder on the set of components of a weighted Toeplitz graph in terms of isomorphism to induced subgraphs.
These results imply that the Frobenius normal form of a symmetric Toeplitz matrix is a direct sum of symmetric irreducible Toeplitz matrices such that for any pair of these blocks, one is a principal submatrix of the other.
Similarly, we find that a component of a weighted Hankel graph is also a weighted Hankel graph, that is, the Frobenius normal form of a Hankel matrix is a direct sum of irreducible Hankel matrices.
However, while a specific preorder exists on the set of components of a weighted Toeplitz graph, this ordering does not extend to the components of a weighted Hankel graph.
%In Section~4, we design a linear time algorithm which determines the number of blocks in the Frobenius normal form of a symmetric Toeplitz matrix.
%Section~3 provides two theorems (Theorems~\ref{thm:bigs} and \ref{thm:main}) necessary to design the algorithm.
%We contract the residue classes of a given graph to obtain a simpler graph, which is then studied using number-theoretic techniques to derive the results.

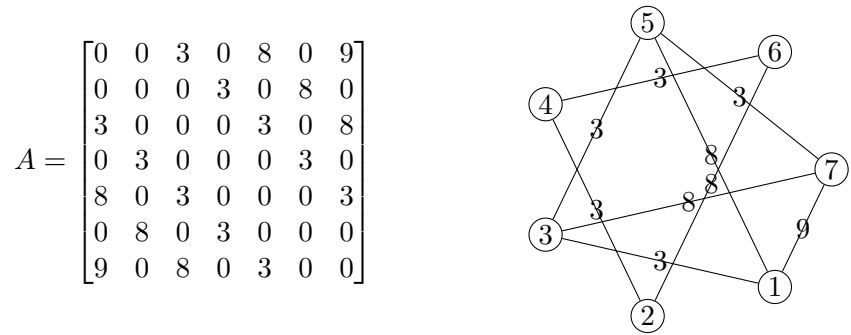
\begin{figure}
\begin{center}
\begin{equation*}
A=\begin{bmatrix}
0 & 0 & 3 & 0 & 8 & 0 & 9 \\
0 & 0 & 0 & 3 & 0 & 8 & 0 \\
3 & 0 & 0 & 0 & 3 & 0 & 8 \\
0 & 3 & 0 & 0 & 0 & 3 & 0 \\
8 & 0 & 3 & 0 & 0 & 0 & 3 \\
0 & 8 & 0 & 3 & 0 & 0 & 0 \\
9 & 0 & 8 & 0 & 3 & 0 & 0
\end{bmatrix}
\hspace{2cm}
\begin{tikzpicture}[baseline={(0,0)}, scale = 1]
\tikzstyle{vertex} = [circle,draw,fill=white,inner sep=1.3];
 \foreach \i in {1,...,7}{
      \node[vertex] (\i) at (-360*\i/7:2) {$\i$};
      }
\draw[black,>=stealth]  (1) -- (3) node[midway, ] {$3$};
\draw[black,>=stealth]  (2) -- (4) node[midway,] {$3$};
\draw[black,>=stealth]  (3) -- (5) node[midway, ] {$3$};
\draw[black,>=stealth]  (4) -- (6) node[midway, ] {$3$};
\draw[black,>=stealth]  (5) -- (7) node[midway, ] {$3$};
\draw[black,>=stealth]  (1) -- (5) node[midway, ] {$8$};
\draw[black,>=stealth]  (2) -- (6) node[midway, ] {$8$};
\draw[black,>=stealth]  (3) -- (7) node[midway, ] {$8$};
\draw[black,>=stealth]  (1) -- (7) node[midway, ] {$9$};

\end{tikzpicture}
\end{equation*}
\caption{The symmetric Toeplitz matrix $A= T[0,0,3,0,8,0,9]$ and its graph}\label{fig:wtex}
\end{center}
\end{figure}

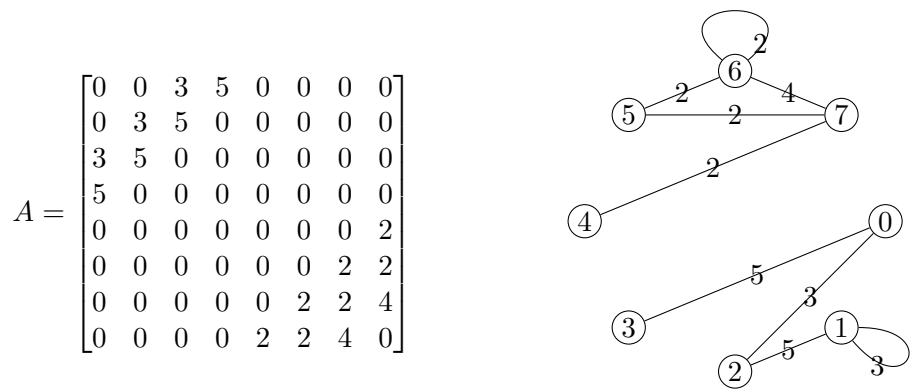
\begin{figure}
\begin{center}
\begin{equation*}
A=\begin{bmatrix}
0 & 0 & 3 & 5 & 0 & 0 & 0 & 0\\
0 & 3 & 5 & 0 & 0 & 0 & 0 & 0\\
3 & 5 & 0 & 0 & 0 & 0 & 0 & 0\\
5 & 0 & 0 & 0 & 0 & 0 & 0 & 0\\
0 & 0 & 0 & 0 & 0 & 0 & 0 & 2\\
0 & 0 & 0 & 0 & 0 & 0 & 2 & 2\\
0 & 0 & 0 & 0 & 0 & 2 & 2 & 4\\
0 & 0 & 0 & 0 & 2 & 2 & 4 & 0
\end{bmatrix}
\hspace{2cm}
\begin{tikzpicture}[baseline={(0,0)}, scale = 1]
\tikzstyle{vertex} = [circle,draw,fill=white,inner sep=1.3];
 \foreach \i in {0,...,7}{
      \node[vertex] (\i) at (-360*\i/8:2) {$\i$};
      }
\draw[black,>=stealth]  (0) -- (2) node[midway, ] {$3$};
\draw[black,>=stealth]  (0) -- (3) node[midway,] {$5$};
\draw[black,>=stealth]  (1) edge [out=0,in=-50,looseness=15] (1) node[below right=7] {$3$} ;
\draw[black,>=stealth]  (1) -- (2) node[midway, ] {$5$};
\draw[black,>=stealth]  (4) -- (7) node[midway, ] {$2$};
\draw[black,>=stealth]  (5) -- (6) node[midway, ] {$2$};
\draw[black,>=stealth]  (5) -- (7) node[midway, ] {$2$};
\draw[black,>=stealth]  (6) edge [out=40,in=140,looseness=10] (6) node[above right=3 ] {$2$};
\draw[black,>=stealth]  (6) -- (7) node[midway, ] {$4$};

\end{tikzpicture}
\end{equation*}
\caption{The Hankel matrix $A= H[0,3,0,5,0,0,0,0,0,0,0,2,2,4,0]$ and its graph}\label{fig:hankel}
\end{center}
\end{figure}

\section{Preliminaries}

For graph-theoretical terms and notations not defined in this paper, we follow \cite{bondy2010graph}.

The \emph{adjacency matrix} of a weighted graph $(G,w)$ is a square matrix $A$ of size $n \times n$ such that its element $a_{ij}$ is equal to the weight of an edge joining $v_i$ and $v_j$ when it exists, and zero when there is no edge joining $v_i$ and $v_j$. 

A square matrix $A$ is called \emph{reducible} if by a simultaneous permutation of its lines we can obtain a matrix of the form
\[
\begin{bmatrix}
A_{1} & A_{12} \\
O & A_{2}
\end{bmatrix}
\]
where $A_1$ and $A_2$ are square matrices of order at least one. 
If $A$ is not reducible, then $A$ is called \emph{irreducible}. 
The following theorem gives a relation between irreducibility of a matrix and its graph. 

\begin{theorem}\label{thm:irreducible_connect}
Let $A$ be a square matrix of order $n$. 
Then $A$ is irreducible if and only if the (weighted) graph having adjacency matrix $A$ is connected. 
\end{theorem}

Let $A$ be a square matrix of order $n$. 
Then there exists a permutation matrix $P$ of order $n$ and an integer $t \ge 1$ such that 
\begin{equation}\label{eq:PNF}
P A P^T = 
\begin{pmatrix}
A_1 & A_{12} & \cdots & A_{1t} \\
O & A_2 & \cdots &A_{2t} \\
\vdots & \vdots & \ddots & \vdots \\
O & O & \cdots & A_t
\end{pmatrix}
\end{equation}
where $A_1, A_2, \ldots, A_t$ are square irreducible matrices. 
The matrices in equation~\eqref{eq:PNF} that occur as diagonal blocks are uniquely determined up to simultaneous permutation of their lines. 
The form in equation~\eqref{eq:PNF} is called the \emph{Frobenius normal form} of the square matrix $A$. 
If $A$ is symmetric, then its Frobenius normal form is as follows:
\begin{equation}\label{eq:SYM}
P A P^T = 
\begin{pmatrix}
A_1 & O & \cdots &O \\
O & A_2 & \cdots &O \\
\vdots & \vdots & \ddots & \vdots \\
O & O & \cdots & A_t
\end{pmatrix}
\end{equation}
where $A_1, A_2, \ldots, A_t$ are symmetric irreducible matrices. 

Given a weighted graph $(G,w)$, a maximal connected subgraph of $(G,w)$ is called a \emph{component} of $(G,w)$.
Then, by Theorem~\ref{thm:irreducible_connect}, it is well-known that each of $A_1, A_2, \ldots, A_t$ in equation~\ref{eq:SYM} equals the adjacency matrix of a component of the graph $(G,w)$ whose adjacency matrix is $A$.

\section{The Frobenius normal form of a symmetric Toeplitz matrix}

\begin{lemma}\label{lem:diff}
Let $A = T[a_0, \ldots, a_{n-1}]$. 
Suppose that there are vertices $u_1$, $u_2$, $v_1$, and $v_2$ in a component $C$ of $G(A)$ such that $u_1-v_1=u_2-v_2  \in S_A$.
Then \[\left|V(C) \cap \llbracket v_1,u_1 \rrbracket \right| = \left|V(C)\cap \llbracket v_2,u_2 \rrbracket \right|.\] 
\end{lemma}

\begin{proof}
	Without loss of generality, assume $u_2 \ge u_1$.
	Take $a \in V(C) \cap \llbracket v_1,u_1-1 \rrbracket$.
	Since $u_1-v_1=u_2-v_2=:s$, there is a unique integer $b$ in $\llbracket v_2,u_2-1 \rrbracket$ such that $s \mid b-a$.
Thus, since $u_2 \ge u_1$, $b = a+ks$ for some nonnegative integer $k$. 
Since $s\in S_A$, there is a walk $a(a+s) \cdots (a+ks) = b$ in $G(A)$.
	Thus $a$ and $b$ are in the same component $C$.
	Then $b \in V(C) \cap \llbracket v_2,u_2-1 \rrbracket$.
	Therefore $a\mapsto b$ is a function from $V(C) \cap \llbracket v_1, u_1-1 \rrbracket$ to $V(C) \cap \llbracket v_2,u_2-1\rrbracket$.
Suppose $a_1 \mapsto b$ and $a_2 \mapsto b$.
Then $s \mid b-a_1$ and $s\mid b-a_2$ and so $s \mid a_1-a_2$.
Since $|a_1-a_2| \le s-1$, $a_1 = a_2$. 
Thus the given function is injective. 
Therefore 
\[|V(C) \cap \llbracket v_1, u_1-1\rrbracket| \le |V(C) \cap \llbracket v_2,u_2-1\rrbracket|. \]
	By a similar argument, one may show that there is an injective function from $V(C) \cap \llbracket v_2,u_2-1\rrbracket$ to $V(C) \cap \llbracket v_1, u_1-1\rrbracket $ and so
	$|V(C) \cap \llbracket v_1, u_1-1\rrbracket| \ge |V(C) \cap \llbracket v_2,u_2-1\rrbracket|$. 
	Thus 	\[|V(C) \cap \llbracket v_1, u_1-1\rrbracket| = |V(C) \cap \llbracket v_2,u_2-1\rrbracket|. \]
	Since $u_1$ and $u_2$ are in $C$,
\[|V(C) \cap \llbracket v_1, u_1\rrbracket|=|V(C) \cap \llbracket v_1, u_1-1\rrbracket|+1=|V(C) \cap \llbracket v_2,u_2-1\rrbracket|+1=|V(C) \cap \llbracket v_2,u_2 \rrbracket |.  \qedhere \]
\end{proof}

The vertex set of a component of a weighted Toeplitz graph may not have consecutive labels of the vertices. 
When the vertex set of a component is labeled as $n_1<n_2< \cdots< n_k$, relabeling it as $1, 2, \ldots, k$ is called {\it normalized labeling}.
Normalized labeling can be considered as a bijective map.

\begin{theorem}\label{thm:sqz}
Each component of a weighted Toeplitz graph is isomorphic to a weighted Toeplitz graph under the normalized labeling of its vertex set.
\end{theorem}

\begin{proof}
Let $A = T[a_0, \ldots, a_{n-1}]$. 
Take a component $C$ of $G(A)$.
Let $\psi:V(C) \rightarrow [t]$ be the normalized labeling of $C$ where $t = |V(C)|$, and let 
\[ D = \{s \in S_A \colon\, \exists u,v \in V(C) \text{ s.t. } u-v = s\}.\]
We claim that $\psi$ is an isomorphism which naturally induces a weighted Toeplitz graph $G(B)$ for some symmetric Toeplitz matrix $B=T[b_0,\ldots, b_{t-1}]$.
To determine $b_0, \ldots, b_{t-1}$ from $\psi$, 
we consider the function 
\[\varphi: D \rightarrow \mathbb{Z}\] 
sending $s \in D$ to $\psi(v+s)-\psi(v)$ for some vertex $v$ in $C$ with $v+s$ belonging to $C$. 
We will show that $\varphi$ is well-defined. 
Take $s \in D$.
First, we note that there is a vertex $v$ in $C$ with $v+s$ belonging to $C$ by the definition of $D$. 
Further, $|V(C)\cap \llbracket v,v+s \rrbracket|$ means the number of vertices in $C$ between $v$ and $v+s$ (including $v$ and $v+s$) when the vertices are arranged in increasing order.
Therefore, by the definition of normalized labeling, \[\psi(v+s)-\psi(v) = |V(C)\cap \llbracket v,v+s\rrbracket|-1\] for any $v \in V(C)$ with $v+s$ belonging to $C$.
Thus, by Lemma~\ref{lem:diff}, 
\[
\psi(v_1+s) - \psi(v_1) = |V(C)\cap \llbracket v_1,v_1+s\rrbracket|-1 = |V(C)\cap \llbracket v_2,v_2+s\rrbracket|-1 = \psi(v_2+s) - \psi(v_2)
\]
for any vertices $v_1$ and $v_2$ in $C$ with both $v_1+s$ and $v_2+s$ belonging to $C$.
Consequently, we have shown that $\varphi$ is well-defined. 

To show the injectivity of $\varphi$, suppose that $\varphi(s_1) = \varphi(s_2)$ for some $s_1, s_2 \in D$ with $s_1 \le s_2$.
Since $s_2 \in D$, there is a vertex $v \in V(C)$ such that $v+s_2 \in V(C)$. 
Then $v+s_1 \le v+s_2$ and so $v$ and $v+s_1$ are adjacent.
Thus $v+s_1 \in V(C)$. 
By the definition of $\varphi$, $\psi(v+s_i)-\psi(v) = \varphi(s_i)$ for each $i = 1, 2$. 
Then 
\[
\psi(v+s_1)-\psi(v) = \varphi(s_1) = \varphi(s_2) = \psi(v+s_2) - \psi(v)
\] by the assumption.  
Thus $\psi(v+s_1) = \psi(v+s_2)$.
Since a normalized labeling is bijective, $v+s_1=v+s_2$ and so $s_1=s_2$.
Therefore $\varphi$ is injective.

Now, we will show that $\psi$ induces an isomorphism from $C$ to $G(B)$ where $B:=T[b_0, \ldots, b_{t-1}]$ such that
\[b_i=\begin{cases}
	0 & \text{if $i \notin \varphi(D)$;} \\
	a_s & \text{if $i= \varphi(s)$ for some $s\in D$.} 
\end{cases} \]
Then $S_B= \{ i \in \llbracket 0, t-1 \rrbracket \colon\, b_i \neq 0 \} = \{\varphi (s) \colon\, s \in D\} =\varphi(D)$.
Therefore \[S_B = \varphi(D).\]

Take an edge $uv$ in $C$. 
Without loss of generality, we may assume $u \ge v$.
Then $u-v = s$ for some $s \in D$.
By the definition of $\varphi$, $\psi(u)-\psi(v) = \varphi(s) \in \varphi(D) = S_B$.
Thus there is an edge $\psi(u)\psi(v)$ in $G(B)$ with weight $b_{\varphi(s)}$ by the definition of a weighted Toeplitz graph. 
Since $b_{\varphi(s)}=a_s$, the weight of $\psi(u)\psi(v)$ and the weight of $uv$ are the same.

Take an edge $ab$ in $G(B)$ with $a \ge b$.
Then there exists $s \in D$ such that $a-b = \varphi(s)$. 
We wish to show that $\psi^{-1}(a) -\psi^{-1}(b) = s$. 
Since $s\in D$, there exist some vertices $u$ and $v$ in $C$ satisfying $u-v=s$.
Then $\psi(u)-\psi(v) = \varphi(s)$ by the definition of $\varphi$. 
Therefore $\psi(u)-\psi(v) = a-b$ and so $\psi(u)-a = \psi(v)-b$.
Thus either $\psi(u) \le a$ or $\psi (v) > b$. 

We first suppose $\psi(u) \le a$. 
By the definition of normalized labeling, $u \le \psi^{-1}(a)$ and so $v=u-s \le \psi^{-1}(a) - s$. 
Thus $\psi^{-1}(a) -s$ is a vertex in $G(A)$.
Since $\psi^{-1}(a) \in V(C)$ and $s\in D \subseteq S_A$, $\psi^{-1}(a)-s \in V(C)$. 
Since $s= \psi^{-1}(a) - (\psi^{-1}(a)-s)$, 
\[\varphi(s) = a-\psi(\psi^{-1}(a)-s)\] by the definition of $\varphi$.
Therefore, since $a-b=\varphi(s)$, $\psi(\psi^{-1}(a)-s) = a - \varphi(s) = b$ and so $\psi^{-1}(a)-s = \psi^{-1}(b)$.
Hence $\psi^{-1}(a)-\psi^{-1}(b) = s$ and $\psi^{-1}(a)\psi^{-1}(b)$ is an edge of $C$.

Now, suppose $\psi(v) > b$. 
Then $v > \psi^{-1}(b)$ and so $u =v+s> \psi^{-1}(b) +s$. 
Thus $\psi^{-1}(b)+s$ is a vertex in $G(A)$.
Since $\psi^{-1}(b) \in V(C)$ and $s\in D \subseteq S_A$,  $\psi^{-1}(b)+s \in V(C)$. 
Since $s = (\psi^{-1}(b)+s)- \psi^{-1}(b)$, \[\varphi(s) = \psi(\psi^{-1}(b)+s) - b\] by the definition of $\varphi$. 
Therefore $\psi(\psi^{-1}(b)+s) = b + \varphi(s) = a$ and so $\psi^{-1}(b) + s = \psi^{-1}(a)$. 
Hence $\psi^{-1}(a)-\psi^{-1}(b) = s$ and $\psi^{-1}(a)\psi^{-1}(b)$ is an edge of $C$.

For both cases,  $\psi^{-1}(a)\psi^{-1}(b)$ is an edge of $C$ satisfying $\psi^{-1}(a)-\psi^{-1}(b) = s$. 
Moreover, since $a_s = b_{\varphi(s)}$, the weight of $\psi^{-1}(a)\psi^{-1}(b)$ in $C$ equals the weight of $ab$ in $G(B)$.  

We have shown that $\psi$ induces an isomorphism from $C$ to the weighted Toeplitz graph $G(B)$. 
Since $C$ was arbitrarily chosen, each of components of $G$ is isomorphic to a weighted Toeplitz graph under the normalized labeling of its vertex set. 
\end{proof}

Let $A$ be a symmetric Toeplitz matrix of order $n$.
By the definition of a weighted Toeplitz graph, for any vertices $u$ and $v$ in $G(A)$ and any integer $i$ with $\{u+i, v+i\} \subseteq [n]$, in $G(A)$, $u$ and $v$ are adjacent if and only if $u+i$ and $v+i$ are adjacent.
Further, the edge $(u+i)(v+i)$ has the same weight as the edge $uv$.
Thus we obtain the following lemma.

\begin{lemma}\label{lem:translate}
Let $A = T[a_0, \ldots, a_{n-1}]$ and $C$ be a component of $G(A)$.
Then, for any integer $i$ with $T_C(i):=\{v+i \colon\, v\in V(C) \} \subseteq [n]$, the vertices in $T_C(i)$ form a subgraph of $G(A)$ isomorphic to $C$.
\end{lemma}

By the previous lemma, we may give a preorder $\lesssim$ on the set of components of a weighted Toeplitz graph as follows.

\begin{theorem}\label{thm:supercomp}
Let $A = T[a_0, \ldots, a_{n-1}]$. 
In addition, let $C_1, \ldots, C_k$ be the components of $G(A)$ such that $M(C_1) < M(C_2) < \cdots < M(C_k)$
where $M(C_i):=\max \{v \colon\, v\in V(C_i)\} $ for each $i \in [k]$.
Then
\[C_1 \lesssim C_2 \lesssim \cdots \lesssim C_k \]
where $G \lesssim H$ means ``$G$ is isomorphic to an induced subgraph of $H$''.
\end{theorem}
\begin{proof}
	Take two components $C$ and $C'$ such that $M(C) < M(C')$.
    We note that \[M(C') \in T_C(M(C')-M(C))=\{v+M(C')-M(C) \colon\, v\in V(C)\}.\]
    Further, the subgraph $H$ of $G(A)$ induced by $T_C(M(C')-M(C))$ is isomorphic to $C$ which is connected by Lemma~\ref{lem:translate}.
    Since $C'$ is a component and $M(C')\in V(H)$, $H$ is isomorphic to an induced subgraph of $C'$.
\end{proof}

Theorems~\ref{thm:sqz} and \ref{thm:supercomp} can be rewritten in a matrix version as follows.

\begin{theorem}\label{thm:FNF}
For a symmetric Toeplitz matrix $A$, there exists a permutation matrix $P$ such that 
\[
P^TAP = 
\begin{bmatrix}
A_1 & O & \cdots & O \\
O & A_2 & \cdots & O \\
\vdots & \vdots & \ddots & \vdots \\
O & O & \cdots & A_k
\end{bmatrix}
\]
for some positive integer $k$ and irreducible symmetric Toeplitz matrices $A_1, \ldots, A_k$ such that $A_i$ is a principal submatrix of $A_j$ if and only if $i\le j$.
\end{theorem}

\begin{proof}
Let $A=T[a_0, \ldots, a_{n-1}]$ be a symmetric Toeplitz matrix.
Let $C_1, C_2, \ldots, C_k$ be the components of $G(A)$ such that \[C_1 \lesssim C_2 \lesssim \cdots \lesssim C_k \]
by Theorem~\ref{thm:supercomp} ($\lesssim$ is as defined in Theorem~\ref{thm:supercomp}).
Then the vertex set of $G(A)$ is the disjoint union of $V(C_1), V(C_2), \ldots, V(C_k)$.
By Theorem~\ref{thm:sqz}, there exist isomorphisms $\psi_1, \psi_2, \ldots, \psi_k$ so that $\psi_i:V(C_i) \to [n_i]$ is an isomorphism from $C_i$ to a weighted Toeplitz graph $G(A_i)$ for some symmetric Toeplitz matrix $A_i$ where $n_i:=|V(C_i)|$.
Since each of $C_1, C_2, \ldots, C_k$ is connected, $A_1, A_2, \ldots, A_k$ are irreducible.
Now we define a permutation $\sigma$ on $[n]$ as follows. 
For each $v \in V(G(A))$, there exists $i \in [k]$ satisfying $v \in C_i$ and we let 
\[\sigma(v)=\sum_{j=1}^{i-1} n_j+\psi_i(v). \] 
Let $P$ be the permutation matrix representing $\sigma$. 
Then 
\[
P^TAP = 
\begin{bmatrix}
A_1 & O & \cdots & O \\
O & A_2 & \cdots & O \\
\vdots & \vdots & \ddots & \vdots \\
O & O & \cdots & A_k 
\end{bmatrix}.
\]
Since $C_1 \lesssim C_2 \lesssim \cdots \lesssim C_k$, $A_i$ is a principal submatrix of $A_j$ if and only if $i\le j$.
\end{proof}

	Let $A$ be a symmetric Toeplitz matrix.
	Then, by using a graph algorithm, all the components of $G(A)$ could be determined, which correspond to the diagonal blocks of the Frobenius normal form of $A$. 
	Thus one can derive the Frobenius normal form of $A$ as in the proof of Theorem~\ref{thm:FNF}. 
	For example, in Figure~\ref{fig:wtex}, there are two components such that one consists of the vertex set $\{2,4,6\}$ and the other consists of the vertex set $\{1,3,5,7\}$.
	Thus $k=2$ with $A_1=T[0,3,8]$ and $A_2=T[0,3,8,9]$.

\section{The Frobenius normal form of a Hankel matrix}

\begin{lemma}\label{lem:diff_H}
Let $A = H[a_0, a_1, \ldots, a_{2n-2}]$. 
Suppose that there are vertices $u_1$, $u_2$, $v_1$, and $v_2$ in a component $C$ of $G(A)$ such that $u_1+v_1=u_2+v_2  \in T_A$ and $v_2 \le v_1\le u_1 \le u_2$.
Then \[\left|V(C) \cap \llbracket u_1,u_2 \rrbracket \right| = \left|V(C)\cap \llbracket v_2,v_1 \rrbracket \right|.\] 
\end{lemma}

\begin{proof}
	Take $a \in V(C) \cap \llbracket u_1,u_2 \rrbracket$.
	Since $s:=u_1+v_1=u_2+v_2 \in T_A$, $s-a \in \llbracket v_2,v_1 \rrbracket$ and it is adjacent to $a$.
	Thus $s-a \in V(C) \cap \llbracket v_2,v_1 \rrbracket$.
	Therefore $a\mapsto (s-a)$ is a function $f$ from $V(C) \cap \llbracket u_1, u_2 \rrbracket$ to $V(C) \cap \llbracket v_2,v_1 \rrbracket$.
	It is easy to check that $b \mapsto (s-b)$ from $V(C) \cap \llbracket v_2,v_1 \rrbracket$ to $V(C) \cap \llbracket u_1, u_2 \rrbracket$ is the inverse function of $f$.
	Thus \[\left|V(C) \cap \llbracket u_1,u_2 \rrbracket \right| = \left|V(C)\cap \llbracket v_2,v_1 \rrbracket \right|. \qedhere \]
\end{proof}

\begin{theorem}\label{thm:sqz_H}
Each component of a weighted Hankel graph is isomorphic to a weighted Hankel graph under the normalized labeling of its vertex set.
\end{theorem}

\begin{proof}
Let $A = H[a_0, a_1, \ldots, a_{2n-2}]$. 
Take a component $C$ of $G(A)$.
Let $\psi:V(C) \rightarrow \llbracket 0, k-1 \rrbracket$ be the normalizing labeling of $C$ where $k = |V(C)|$, and 
\[ S = \{t \in T_A \colon\, \exists u,v \in V(C) \text{ s.t. } u+v = t\}.\]
We claim that $\psi$ is an isomorphism which naturally induces a weighted Hankel graph $G(B)$ for some Hankel matrix $B=H[b_0,\ldots, b_{2k-2}]$.
To determine $b_0, \ldots, b_{2k-2}$ from $\psi$, 
we consider the function 
\[\varphi: S \rightarrow \mathbb{N}\] 
sending $t \in S$ to $\psi(v)+\psi(t-v)$ for some vertex $v$ with $v$ and $t-v$ both belonging to $C$. 
We will show that $\varphi$ is well-defined. 
First, we note that for each $t \in S$, there is a vertex $v$ with $v$ and $t-v$ both belonging to $C$ by the definition of $T$ and $v\le t-v$. 
Suppose that there are two vertices $u$ and $v$ such that $u$, $v$, $t-u$, and $t-v$ are belonging to $C$ for some $t\in S$.
Without loss of generality, we may assume that $u \le v \le t-v \le t-u$.
Since $\psi(a)-\psi(b)$ means the number of vertices between $a$ and $b$ for some $a,b \in V(C)$ with $b \le a$, that is, $|V(C) \cap \llbracket b,a\rrbracket|$,
\begin{align*}
	\psi(t-u)-\psi(v)&=(\psi(t-u)-\psi(t-v))-(\psi(v)-\psi(u))+(\psi(t-v)-\psi(u)) \\ 
	&= |V(C)\cap \llbracket t-v,t-u\rrbracket|+|V(C)\cap \llbracket u,v\rrbracket|+(\psi(t-v)-\psi(u)) \\
	&=\psi(t-v)-\psi(u),
	\end{align*}
where the last equality holds by Lemma~\ref{lem:diff_H}.
Thus $\psi(u)+\psi(t-u)=\psi(v)+\psi(t-v)$.
Consequently, we have shown that $\varphi$ is well-defined. 

To show the injectivity of $\varphi$, suppose that $\varphi(t_1) = \varphi(t_2)$ for some $t_1, t_2 \in S$ with $t_1 \le t_2$.
Since $t_i \in S$, there is a vertex $v_i \in V(D)$ such that $v_i \le t_i-v_i$ and $t_i-v_i \in V(C)$ for each $t=1,2$.
To the contrary, suppose that $v_2 > t_1$.
Then
\[0\le v_1 \le t_1-v_1 \le t_1 <v_2 \le t_2-v_2\le n-1. \]
Thus, by the definition of normalizing labeling, 
\[\psi(v_1) \le \psi(t_1-v_2) < \psi(v_2) \le \psi(t_2-v_2) \]
and so 
\[\varphi(t_1)=\psi(v_1)+\psi(t_1-v_1)< \psi(v_2)+\psi(t_2-v_2)=\varphi(t_2), \]
which contradicts the assumption.
Thus $v_2 \le t_1$.
Then, since $t_2-v_2 \in V(C) \subseteq V(G(A))$, $0\le t_1-v_2 \le t_2-v_2 \le n-1$ and so $t_1-v_2 \in V(G(A))$.
Then
\[
\psi(v_2)+\psi(t_1-v_2) = \varphi(t_1) = \varphi(t_2) = \psi(v) + \psi(t_2-v_2)
\] by the assumption.
Thus $\psi(t_1-v_2)=\psi(t_2-v_2)$.
Since a normalizing labeling is bijective, $t_1-v_2=t_2-v_2$.
Therefore $t_1=t_2$.
Hence $\varphi$ is injective.

Now, we will show that $\psi$ induces an isomorphism from $C$ to $G(B)$ where $B:=H[b_0, \ldots, b_{2k-2}]$ such that
\[b_i=\begin{cases}
	0 & \text{if $i \notin \varphi(S)$;} \\
	a_s & \text{if $i= \varphi(t)$ for some $t\in S$.} 
\end{cases} \]
Then $T_B= \{ i \in \llbracket 0, 2k-2 \rrbracket \colon\, b_i \neq 0 \} = \{\varphi (t) \colon\, t \in S, a_t \neq 0\} =\varphi(S)$.
Therefore \[T_B = \varphi(S).\]

Take an edge $uv$ in $C$. 
Without loss of generality, we may assume $u \ge v$.
Then $u+v = t$ for some $t \in S$.
By the definition of $\varphi$, $\psi(u)+\psi(v) = \varphi(t) \in \varphi(S) = T_B$.
Thus there is an edge $\psi(u)\psi(v)$ in $G(B)$ with weight $b_{\varphi(t)}$ by the definition of a weighted Hankel graph. 
Since $b_{\varphi(t)}=a_t$, the weight of $\psi(u)\psi(v)$ and the weight of $uv$ are the same.

Take an edge $ab$ in $G(B)$ with $a \ge b$.
Then there exists $t \in S$ such that $a+b = \varphi(t)$. 
We wish to show that $\psi^{-1}(a) + \psi^{-1}(b) = t$.
Since $t\in S$, there exist some vertices $v$ and $t-v$ in $C$ satisfying $v \le t-v$.
Then $\psi(v)+\psi(t-v) = \varphi(t)$ by the definition of $\varphi$. 
Therefore $\psi(v)+\psi(t-v) = a+b$.
Thus $\psi(v)   \le b \le a \le \psi(t-v)$ or $b\le \psi(v) \le \psi (t-v)\le a$. 
Suppose that $\psi(v)   \le b \le a \le \psi(t-v)$ (resp.\ $b\le \psi(v) \le \psi (t-v)\le a$).
Then, by the definition of normalizing labeling, 
$v   \le \psi^{-1}(b) \le \psi^{-1}(a) \le t-v$ (resp.\ $\psi^{-1}(b)\le v \le t-v \le \psi^{-1}(a)$).
Thus $0\le (t-v)-\psi^{-1}(a) \le t- \psi^{-1}(a) \le t$ (resp.\ $0\le v-\psi^{-1}(b) \le t- \psi^{-1}(b) \le t$).
Since $\psi^{-1}(a) \in V(C)$ (resp.\ $\psi^{-1}(v) \in V(C)$), $t- \psi^{-1}(a) \in V(C)$ (resp.\ $t- \psi^{-1}(b)\in V(C)$).
Since $t=\psi^{-1}(a)+(t-\psi^{-1}(a))$ (resp.\ $t=\psi^{-1}(b)+(t-\psi^{-1}(b))$),
\[ \varphi(t)=a+\psi(t- \psi^{-1}(a)) \quad (resp.\ \varphi(t)=b+\psi(t- \psi^{-1}(b))) \]
by the definition of $\varphi$.
Therefore $\psi(t- \psi^{-1}(a))=\varphi(t)-a=b$ (resp.\  $\psi(t- \psi^{-1}(b))=\varphi(t)-b=a$) and so $t- \psi^{-1}(a)=\psi^{-1}(b)$ (resp.\ $t- \psi^{-1}(b)=\psi^{-1}(a)$).
Hence $\psi^{-1}(a) + \psi^{-1}(b) = t$ and $\psi^{-1}(a)\psi^{-1}(b)$ is an edge of $C$. 
Moreover, since $a_s = b_{\varphi(s)}$, the weight of $\psi^{-1}(a)\psi^{-1}(b)$ in $C$ equals the weight of $ab$ in $G(B)$.

We have shown that $\psi$ induces an isomorphism from $C$ to the weighted Hankel graph $G(B)$. 
Since $C$ was arbitrarily chosen, each of components of $G$ is isomorphic to a weighted Hankel graph under the normalized labeling of its vertex set. 
\end{proof}

\begin{corollary}
For a Hankel matrix $A$, there exists a permutation matrix $P$ such that 
\[
P^TAP = 
\begin{bmatrix}
A_1 & O & \cdots & O \\
O & A_2 & \cdots & O \\
\vdots & \vdots & \ddots & \vdots \\
O & O & \cdots & A_k
\end{bmatrix}
\]
for some positive integer $k$ and irreducible Hankel matrices $A_1, \ldots, A_k$
\end{corollary}

\begin{proof}
    Let $A=H[a_0, \ldots, a_{2n-2}]$ be a Hankel matrix.
    Let $C_1, C_2, \ldots, C_k$ be the components of $G(A)$.
    Then the vertex set of $G(A)$ is the disjoint union of $V(C_1), V(C_2), \ldots, V(C_k)$.
    By Theorem~\ref{thm:sqz_H}, there exist isomorphisms $\psi_1, \psi_2, \ldots, \psi_k$ so that $\psi_i:V(C_i) \to \llbracket 0, n_i-1 \rrbracket$ is an isomorphism from $C_i$ to a weighted Hankel graph $G(A_i)$ for some Hankel matrix $A_i$ where $n_i:=|V(C_i)|$.
    Since each of $C_1, C_2, \ldots, C_k$ is connected, $A_1, A_2, \ldots, A_k$ are irreducible.
    Now we define a permutation $\sigma$ on $[n]$ as follows. 
    For each $v \in V(G(A))$, $v \in C_i$ for some $i \in [k]$ and we let 
    \[\sigma(v)=\sum_{j=1}^{i-1} n_j+\psi_i(v). \] 
    Let $P$ be the permutation matrix representing $\sigma$. 
    Then 
    \[
    P^TAP = 
    \begin{bmatrix}
    A_1 & O & \cdots & O \\
    O & A_2 & \cdots & O \\
    \vdots & \vdots & \ddots & \vdots \\
    O & O & \cdots & A_k 
    \end{bmatrix}. \qedhere
    \]
    \end{proof}

\begin{remark}
	We showed in Theorem~\ref{thm:supercomp} that for any pair of components of a weighted Toeplitz graph, one is isomorphic to an induced subgraph of the other.
	This guarantees that for any pair of diagonal blocks in the Frobenius normal form of a symmetric Toeplitz matrix, one is a principal submatrix of the other (see Theorem~\ref{thm:FNF}).
	However, in the Frobenius normal form of a Hankel matrix, this does not hold in the same way.
	A counterexample is given in Figure~~\ref{fig:hankel}.
	The Hankel matrix given in Figure~\ref{fig:hankel} is itself in the Frobenius normal form.
	\end{remark}

\section{Acknowledgement}
This paper was completed under the guidance of professor Suh-Ryung Kim.
This work was supported by Science Research Center Program through the National Research Foundation of Korea(NRF) Grant funded by the Korean Government (MSIT)(NRF-2022R1A2C1009648).

\end{document}